\documentclass{amsart}
\usepackage[english]{babel}
\usepackage{amsmath}
\usepackage{amssymb}
\usepackage[dvips]{graphicx}
\usepackage[latin1]{inputenc}
\usepackage[usenames]{color}
\usepackage{amssymb,latexsym}
\newcommand{\R}{\mathbb R}

\newcommand{\N}{\mathbb N}

\newcommand{\Hyp}{\mathbb H}

\newcommand{\Sp}{\mathbb S}

\newtheorem{theorem}{Theorem}
\newtheorem{corollary}{Corollary}
\newtheorem{example}{Example}
\newtheorem{lemma}{Lemma}
\newtheorem{proposition}{Proposition}
\newtheorem{remark}{Remark}

\begin{document}

\title[Eigenvalues of Schrödinger operators on surfaces]{Isoperimetric inequalities for the eigenvalues of natural Schrödinger operators on surfaces}

\author{ Ahmad El Soufi  }

\address{ Laboratoire de Math\'ematiques et Physique Th\'eorique,
UMR CNRS 6083, Universit\'e François Rabelais de Tours, Parc de Grandmont, F-37200
Tours France}
\email{elsoufi@univ-tours.fr}



\begin{abstract}

This paper deals with eigenvalue optimization problems for a family of natural Schrödinger operators arising in some geometrical or physical contexts. These operators, whose potentials are quadratic in curvature, are considered on closed surfaces immersed in space forms and we look for geometries that maximize the eigenvalues. We show that under suitable assumptions on the potential, the first and the second eigenvalues are maximized by (round) spheres.   
\end{abstract} 

\maketitle

\section {Introduction and statement of main results}\label{1}

There has been remarkable interest in recent years in the eigenvalues of Schrödinger operators with quadratic curvature potentials on surfaces.  Operators of this type appear in several contexts where the physics is strongly influenced by the geometry and the topology of the surface. As examples we can cite stability properties of interfaces in reaction-diffusion systems such as Allen-Cahn, see \cite{AFS, Ha, HL}, and the quantum mechanics on
 nanoscale structures, see \cite{daC, DE, EHL}. In this paper we discuss a natural isoperimetric problem for the eigenvalues of such operators on compact immersed surfaces in the $n$-dimensional Euclidean space $\R^n$ or, more generally, in a simply connected space form.  

Let $X:M\to \R^n$ be a regular compact immersed surface of $\R^n$, which means that $M$ is a compact 2-dimensional differentiable manifold and $X$ is a regular (of class $C^2$) immersion. The corresponding Laplace-Beltrami operator  that we denote $-\Delta_X$, admits a positive unbounded sequence of eigenvalues 
  $$0=\lambda_1(-\Delta_{X})  <  \lambda_2(-\Delta_{X}) \le \lambda_3(-\Delta_{X})\le 
 \cdots \le \lambda_i(-\Delta_{X})\le \cdots$$
  
 These eigenvalues have been intensively studied during the last decades. Important results have been proved regarding the existence of a uniform upper bound for $\lambda_i(-\Delta_{X})$ among surfaces of specified topology and area, the behavior of the optimal upper bound in terms of the genus of the surface and the order $i$ of the eigenvalue (see \cite{YY, LY, EI, EI0} for $i=2$ and \cite{K, CE} for $i$ arbitrary),  
 and the existence and the determination of possible maximizing surfaces for the lowest eigenvalues 
(see \cite{H, N2} for surfaces of genus zero, \cite{N1} for tori, \cite{LY} for the case where $M$ is a projective plane, \cite{JNP} and \cite{EGJ}  for the case where $M$ is a Klein bottle and \cite{JLNNP} for orientable surfaces of genus 2)

 
 Let us denote by $\Sp^2$ the standard sphere of radius $1$ naturally embedded in $\R^3$. Hersch \cite{H} proved that the maximum of the first positive eigenvalue $\lambda_2(-\Delta_{X})$ among immersed orientable surfaces $X:M\to \R^n$ of \emph{genus zero} and area $|X(M)|=|\Sp^2|=4\pi$,  is uniquely achieved by $\Sp^2$. 
This isoperimetric property fails  as soon as the genus of $M$ is not zero. Indeed, for any  $M$ such that genus$(M)\ge1$, there exist immersions $X:M\to \R^n$ with $|X(M)|=4\pi$ and $\lambda_2(-\Delta_{X})>\lambda_2(-\Delta_{\Sp^2})$. In fact, according to \cite{CE}, the supremum of $\lambda_k(-\Delta_{X})$ over the set of all $X:M\to R^n$ such that  $|X(M)|=4\pi$, is an nondecreasing function of the genus of $M$ which tends to infinity with a linear growth rate. Therefore, the second eigenvalue  $\lambda_2(-\Delta_{X})$ is not uniformly bounded over the set of all surfaces $X:M\to \R^n$ of fixed area. 
 
 Surprisingly, the situation changes completely as soon as the Laplace-Beltrami operator is penalized by the extrinsic curvature of the surface.
 Indeed, Ilias and the author \cite{EI4} proved that the standard sphere $\Sp^2$ maximizes the first two eigenvalues of the operator $-\Delta_X -2|H_X|^2$ among all compact surfaces $X:M\to\R^n$ of area $4\pi$, where  $|H_X|$ is the length of the mean curvature vector field. The same result is also true for the operator  $-\Delta_X -|h_X|^2$, where $|h_X|$ is the length of the second fundamental form of $X$. A codimension 1 version of this last result was first obtained by Harrell and Loss \cite{HL}. This phenomenon relies on the fact that, whilst the eigenvalues of $-\Delta_X$ are of intrinsic nature, in the sense that they only depend on the Riemannian metric induced on $M$ by $X$ (notice that, since the dimension of the ambient $\R^n$ is not restricted, any Riemannian metric on $M$ can be induced by such an immersion thanks to Nash-Moser Theorem), the spectrum of $-\Delta_X -|h_X|^2$ depends also on how the surface $X(M)$ is bended in  $\R^n$. The results of \cite{EI4} express the fact that in order to induce large eigenvalues of $-\Delta_X $, the surface  $X:M\to R^n$ must be more and more ``bended''.  
 
Now, as observed above, any compact surface $X:M\to \R^n$ with genus$(M)=0$ and $|X(M)|=4\pi$, satisfies both 
$\lambda_2(-\Delta_X)\le \lambda_2(-\Delta_{\Sp^2})$ and  
$\lambda_2(-\Delta_X+|h_X|^2)\le \lambda_2(-\Delta_{\Sp^2}+|h_{\Sp^2}|^2)$. This motivated Harrell and Loss  \cite{HL} to ask if one can interpolate 
to get $\lambda_2(-\Delta_X-\alpha|h_X|^2)\le \lambda_2(-\Delta_{\Sp^2}-\alpha|h_{\Sp^2}|^2)$ for all $\alpha\ge0$.
 In the present article we give a positive answer to this question. Actually, we put this problem into the following more general setting : For all  $\alpha\in\R$ and $\beta\in\R$, we consider the operator  
 $$L_{X,\alpha,\beta}= - \Delta_X -(\alpha |h_X|^2+\beta |H_X|^2).$$
In the particular case of surfaces immersed in $\R^3$ one has
$$L_{X,\alpha,\beta}= - \Delta_X -\alpha (\kappa_1^2+\kappa_2^2)-\frac 1 4\beta (\kappa_1+\kappa_2)^2,$$
where $\kappa_1$ and $\kappa_2$ are the principal curvatures of the surface. In fact, 
any quadratic symmetric polynomial of the principal curvatures can be written as $\alpha |h_X|^2+\beta |H_X|^2$ for a suitable choice of $\alpha$ and $\beta$. 

Recall that any compact totally umbilic surface $X:M\to \R^n$ is a Euclidean 2-sphere of a 3-dimensional linear subspace of $\R^n$. Such a surface will be called ``round sphere''. For a round sphere of area $4\pi$ one has $|h_{\Sp^2}|^2=2|H_{\Sp^2}|^2=2$ and the lowest two eigenvalues of $L_{X,\alpha,\beta}$ are $\lambda_1(L_{\Sp^2,\alpha,\beta})= - 2\alpha -\beta$ and $\lambda_2(L_{\Sp^2,\alpha,\beta})= 2 - 2\alpha -\beta$.
 
Let us start with the following result concerning the first eigenvalue.
 \begin{proposition}\label{lam1} Let $\alpha\in\R$ and $\beta\in\R$ be such that $4\alpha+\beta\ge 0$. For any compact immersed surface $X:M\to \R^n$ such that $|X(M)|=4\pi$, one has 
 \begin{equation}\label{l1}
 \lambda_1(L_{X,\alpha,\beta})\le \lambda_1(L_{\Sp^2,\alpha,\beta})-\varepsilon\alpha\ \text{genus} (M),
  \end{equation}
 where $\varepsilon = 2$ if $M$ is orientable and $\varepsilon=1$ otherwise.\\
 If $4\alpha+\beta>0$, then the equality holds in (\ref{l1}) if and only if $X(M)$ is a round sphere.\\ 
 If $4\alpha+\beta=0$ and $\alpha\neq 0$, then the equality holds in (\ref{l1}) if and only if the surface $M$ endowed with the Riemannian metric induced by $X$ has constant sectional curvature. 
 \end{proposition}
 \begin{remark} 
\begin{itemize}
	\item The assumption  $4\alpha+\beta\ge 0$ is crucial in Proposition \ref{lam1}. For instance, when $\alpha=0$ and $\beta=-1$, it is known that $\Sp^2$ becomes a minimizer of $\lambda_1(- \Delta_X +|H_X|^2)$ among compact surfaces of genus zero immersed in $\R^3$ (see \cite {F}) and it is conjectured that $\Sp^2$ should minimize $\lambda_1(- \Delta_X - \beta |H_X|^2)$ for all $\beta\in (-1, 0)$.
	\item The case where $4\alpha+\beta=0$ corresponds to the situation where the operator $L_{X,\alpha,\beta}$ is intrinsic (i.e. depends only on the Riemannian metric induced on $M$ by $X$). Indeed, thanks to the Gauss equation $|h_X|^2=-2K_X+4|H_X|^2$, we have in this case $L_{X,\alpha,-4\alpha}= - \Delta_X +2\alpha K_X$, where $K_X$ is the Gaussian curvature of $M$ with respect to the metric induced by $X$.
	\item Regarding the equality case in (\ref{eq 0}) when $4\alpha+\beta=0$, it is well known that in $\R^3$, round spheres are  the only immersed surfaces of constant sectional curvature. However, there exist examples of isometrically immersed spheres which are not round. For instance, the  map $X(x,y,z)=(\cos x, \sin x, y, z)$ induces an isometric embedding from the standard sphere into $\R^4$ whose image is not  totally umbilic.
\end{itemize}
 \end{remark} 
 Regarding the second eigenvalue, we first discuss the case of surfaces of genus zero (i.e. topologically equivalent to a sphere).
 \begin{theorem}\label{genus=0} Let $M$ be a compact surface of genus zero and let $\alpha\in\R$ and $\beta\in\R$ be such that $4\alpha+\beta\ge 0$. For any immersion $X:M\to \R^n$ such that $|X(M)|=4\pi$, one has   
 \begin{equation}\label{eq 0}
 \lambda_2(L_{X,\alpha,\beta})\le \lambda_2(L_{\Sp^2,\alpha,\beta}).
  \end{equation}
 If $4\alpha+\beta>0$, then the equality holds in (\ref{eq 0}) if and only if $X(M)$ is a round sphere.\\ 
 \end{theorem}
  
 The first part of this theorem extends the result of Harrell \cite{Ha} concerning the special case where $n=3$, $\alpha\ge 0$ and $\beta\ge -2\alpha$. Indeed, these conditions on $\alpha$ and $\beta$ are equivalent to the non-negativity of $\alpha |h_X|^2+\beta |H_X|^2=\alpha (\kappa_1^2+\kappa_2^2)+\frac 1 4\beta (\kappa_1+\kappa_2)^2$ as a quadratic form in the principal curvatures $\kappa_1$ and $\kappa_2$.

 For surfaces of higher genus, we obtain the following
 
 \begin{theorem}\label{genus=arb} 
    Let $M$ be a compact orientable surface of genus $\gamma\ge 1$ and let $\alpha\in\R$ and $\beta\in\R$ be such that $\alpha\ge 0$ and $4\alpha+\beta\ge 0$. If either 
   \begin{enumerate}
	\item [i)] $\gamma$ is even and $8\alpha +\beta > 2$, or
	\item [ii)]  $\gamma$  is odd and $4\frac {2\gamma +1}{\gamma +1}\alpha+\beta> 2 $,
\end{enumerate}
then, for any immersion $X:M\to R^n$ such that $|X(M)|=4\pi$ one has
 $$\lambda_2(L_{X,\alpha,\beta})< \lambda_2(L_{\Sp^2,\alpha,\beta}).$$
 \end{theorem}
 A similar result can be obtained for non-orientable surfaces (see Proposition \ref{prop2}).

Let us come back to the one-parameter family of operators 
$$L_{X,\alpha}:=L_{X,\alpha,0}= - \Delta_X -\alpha |h_X|^2.$$ 
 As observed above, for any non-spherical surface $X:M\to R^n$ of area $4\pi$, the function $\alpha\mapsto\lambda_2(L_{\Sp^2,\alpha})-\lambda_2(L_{X,\alpha})$ is always positive for $\alpha=1$ and often negative for $\alpha=0$. A natural question is to determine the bifurcation value for $\alpha$ that we define as follows:
 $$\alpha_X:=\inf\{\alpha_0\ge 0 \, ; \, \lambda_2(L_{X,\alpha})\le \lambda_2(L_{\Sp^2,\alpha}),\, \forall\, \alpha\ge \alpha_0\},$$
or, in order to relax the condition  $|X(M)|=4\pi$, 
$$\alpha_X:=\inf\{\alpha_0\ge 0 \, ; \, \lambda_2(L_{X,\alpha})|X(M)|\le 4\pi \lambda_2(L_{\Sp^2,\alpha}),\, \forall\, \alpha \ge \alpha_0\},$$
An immediate consequence of Theorem \ref{genus=0} and Theorem \ref{genus=arb} above is the following
 \begin{corollary}\label{cor}Let $X:M\to R^n$ be a compact immersed orientable surface. 
\begin{enumerate}
\item [i)] If the genus of $M$  is zero, then $\alpha_X=0$
	\item [ii)] If the genus of $M$  is even, then $\alpha_X\le \frac 1 4$ 
	\item [iii)] If the genus $\gamma$ of $M$ is odd, then $\alpha_X\le \frac {\gamma+1} {4\gamma +2}$.
\end{enumerate}
\end{corollary}
 Assertion (i) of this corollary answers the question of Harrell and Loss \cite{HL}.
\begin{example} For surfaces of genus 1 (immersed tori), Corollary \ref{cor} gives: $\alpha_X\le \frac 1 3$. Let us compute the exact value of $\alpha_X$ for two particular immersed tori : 
\begin{itemize}
	\item [-]
the Clifford torus $\mathbb{T}:=\Sp^1(\frac 1 {\sqrt 2})\times \Sp^1(\frac 1 {\sqrt 2})$ naturally embedded in $\R^4$, for which $|\mathbb{T}|=2\pi^2$, $|h_\mathbb{T}|^2=4|H_\mathbb{T}|^2=4$ and $\lambda_2(-\Delta_\mathbb{T})= 2$ (indeed, $\mathbb{T}$ is a flat minimal surface in $\Sp^3$). Thus, 
$\lambda_2(L_{\mathbb{T},\alpha})|\mathbb{T}|=4\pi^2(1-2\alpha)$. Comparing with  $\lambda_2(L_{\Sp^2,\alpha}) \ |\Sp^2|=8\pi (1-\alpha)$,  we get
$$\alpha_\mathbb{T} = \frac{\pi-2}{2(\pi-1)}\approx 0.2665,$$
	\item [-] the equilateral torus  $M={\mathbb R}^2 / {\mathbb Z}(1,0) \oplus {\mathbb Z}({\frac 1
2},{\sqrt{3}\over 2})$ embedded in $\Sp^5\subset \R^6$ by
$$X(x, y)= \frac{1}{\sqrt 3} (e^ {4i\pi y/\sqrt 3}, e^{2i\pi(x-y/\sqrt 3)}, e^{ 2i\pi (x + y/\sqrt 3)}).$$ 
Again, $X(M)$ is a flat minimal surface in $\Sp^5$ with $|X(M)| =\frac {4\pi^2}{\sqrt 3}$, $|h_X|^2=4|H_X|^2=4$ and $\lambda_2(-\Delta_X)= 2$. This yields
$\lambda_2(L_{X,\alpha})\ |X(M)| =\frac {8\pi^2}{\sqrt 3}(1-2\alpha)$ and, then,
$$\alpha_X= \frac{\pi-\sqrt 3}{2\pi-\sqrt 3}\approx 0.3097 .$$
\end{itemize}
\end{example}

All the results above can be transposed in a more general setting. Indeed, for any $c\in\R$, let $N^n(c)$ be the simply connected space form of curvature $c$ and dimension $n$ (that is $N^n(c)$ is isometric to $\R^n$ for $c=0$, the standard sphere $\Sp^n$ for $c=1$ and the hyperbolic space $\Hyp^n$ for $c=-1$). 
For any immersed surface $X:M\to N^n(c)$, we still denote by $\Delta_X$ the corresponding Laplace-Beltrami operator and by $h_X$ and $H_X$ the second fundamental form and the mean curvature vector, respectively. As above, we consider the operator
 $$L_{X,\alpha,\beta}= - \Delta_X -(\alpha |h_X|^2+\beta |H_X|^2).$$
Among the immersed surfaces $X:M\to N^n(c)$ of $N^n(c)$, the geodesic spheres of $N^3(c)$ play a particular role. These spheres are totally umbilic and any compact totally umbilic surface of $N^n(c)$ is a geodesic sphere of $N^3(c)\hookrightarrow\N^n(c)$.

\begin{theorem}\label{N(c)} 
    Let $X:M\to N^n(c)$ be a compact surface immersed in $N^n(c)$ whose genus and area are denoted by $\gamma$ and $a$, respectively. Let $\alpha\in\R$ and $\beta\in\R$ be such that $4\alpha+\beta\ge 0$.\\ 
 (I)  One has
   \begin{equation}\label{l1c}
 \lambda_1(L_{X,\alpha,\beta})\le \lambda_1(L_{S(a),\alpha,\beta})-\frac{4\varepsilon\pi}{a}\alpha \gamma,
  \end{equation}
  where $S(a)$ is a geodesic sphere of area $a$ in $N^3(c)$, $\varepsilon = 2$ if $M$ is orientable and $\varepsilon=1$ otherwise. Moreover, if $4\alpha+\beta>0$, then the equality holds in (\ref{l1c}) if and only if $X(M)$ is a totally umbilic 2-sphere. If $4\alpha+\beta=0$ and $\alpha\neq 0$, then the equality holds in (\ref{l1c}) if and only if $M$ endowed with the metric induced by $X$ has constant sectional curvature. 
  
\noindent (II) Assume furthermore that $M$ is orientable and that one of the following holds
   \begin{enumerate}
   \item [i)] $\gamma=0$, 
	\item [ii)] $\gamma$ is even, $\alpha\ge 0$ and $8\alpha +\beta \ge 2$, 
	\item [iii)]  $\gamma$  is odd, $\alpha\ge 0$ and $4\frac {2\gamma +1}{\gamma +1}\alpha+\beta\ge 2 $,
\end{enumerate}
then 
\begin{equation}\label{l2c}
\lambda_2(L_{X,\alpha,\beta})\le \lambda_2(L_{S(a),\alpha,\beta}).
\end{equation}
Moreover, if $\gamma=0$ and $4\alpha+\beta>0$, then the equality holds in (\ref{l2c}) if and only if $X(M)$ is a totally umbilic 2-sphere.  
 \end{theorem}
 
For a similar result concerning non-orientable surfaces, see Proposition \ref{prop2}.
 
 \section {Proof of results}

 Let $M$ be a 2-dimensional compact manifold and let $X:M\to N^n(c)$ be a regular immersion (of class $C^2$). In all the sequel we denote by $h_X$ the second fundamental form and by $H_X:=\frac 1 2 \mbox{trace} h_X$ the mean curvature vector field of the immersed surface $X:M\to N^n(c)$. 
  For any  $\alpha\in\R$ and $\beta\in\R$, we consider on $M$ the operator 
 $$L_{X,\alpha,\beta}= - \Delta_X -(\alpha |h_X|^2+\beta |H_X|^2)$$
 whose spectrum consists of a nondecreasing and unbounded sequence of eigenvalues
  $$\lambda_1(L_{X,\alpha,\beta})  <  \lambda_2(L_{X,\alpha,\beta}) \le \lambda_3(L_{X,\alpha,\beta})\le 
 \cdots \le \lambda_i(L_{X,\alpha,\beta})\le \cdots$$
 Geodesic spheres $S(a)$ of area $a$ in $N^3(c)$ are totally umbilic and have constant Gauss curvature so that $K_{S(a)}=\frac{4\pi} a$ and $|h_{S(a)}|^2=2|H_{S(a)}|^2=2\left(\frac{4\pi} a - c\right)$. Hence,
 $$\lambda_1(L_{S(a),\alpha,\beta})= -(2\alpha+\beta)\left(\frac{4\pi} a - c\right)$$
 and
 $$\lambda_2(L_{S(a),\alpha,\beta})= \frac{8\pi} a -(2\alpha+\beta)\left(\frac{4\pi} a - c\right).$$
It is well known that the minimum of the functional $\int_M |H_X|^2$ among immersed surfaces of $\R^n$ is uniquely achieved by round spheres (see \cite{ch}). This result can be easily generalized as follows 
  \begin{lemma}\label{chen} For any compact surface $X:M\to N^n(c)$ one has
 \begin{equation}\label{ch}\int_M \left( |H_X|^2+c\right)v_X\ge 4\pi
 \end{equation}
 where $v_X$ is the volume element induced by $X$ on $M$. Moreover, the equality  
 holds in (\ref{ch}) if and only if $X(M)$ is a totally umbilic 2-sphere.   
 \end{lemma}
 \begin{proof}
  It is well known that the integral $\int_M( |h_X|^2- 2|H_X|^2)v_X$ does not vary under conformal changes of metric in the ambient space. Since the manifold $N^n(c)$ (minus a point if $c>0$) is conformally equivalent to $\R^n$, the surface $X(M)$ can be seen as immersed in $\R^n$ endowed with a metric of constant curvature $c$ conformal to the standard one. We denote by $\overline{X}$ the immersion $X$ considered as an immersion into the standard $\R^n$. Hence, 
 \begin{equation}\label{omb}
 \int_M ( |h_X|^2- 2|H_X|^2)v_X=\int_M( |h_{\overline{X}}|^2- 2|H_{\overline{X}}|^2)v_{\overline{X}}.
 \end{equation}
From Gauss equation, we get 
$$ |h_X|^2- 2|H_X|^2=2 \left(|H_X|^2 +c-K_X\right)$$
and
$$ |h_{\overline{X}}|^2 - 2|H_{\overline{X}}|^2=2 \left(|H_{\overline{X}}|^2 -K_{\overline{X}}\right).$$
Since $\int_M K_{\overline{X}}v_{\overline{X}}=\int K_X v_X$ (Gauss-Bonnet theorem), we obtain after integration and identification, 
$$\int_M \left( |H_X|^2+c\right)v_X=\int_M |H_{\overline{X}}|^2 v_{\overline{X}}\ge 4\pi,$$
where the last inequality is the classical Willmore-Chen inequality \cite{ch}. 

Now, the equality in (\ref{ch}) implies the equality $\int_M |H_{\overline{X}}|^2 v_{\overline{X}}= 4\pi$ which holds if and only if the immersed surface $\overline{X}:M\to\R^n$ is a  round sphere in $\R^n$ (see \cite{ch}). Using (\ref{omb}), we deduce that the surface $X:M\to N^n(c)$ is totally umbilic in $N^n(c)$.   
\end{proof}
\begin{proof}[Proof of Proposition \ref{lam1} and Theorem \ref{N(c)} (I) ] 
 Using constant functions as trial functions in the Rayleigh quotient we get 
 $$\lambda_1(L_{X,\alpha,\beta})\le \frac 1 {|X(M)|}\left( -\alpha \int_M|h_X|^2-\beta \int_M|H_X|^2\right).$$ 
  Gauss equation yields
 $$|h_X|^2=-2K_X+4|H_X|^2+2c.$$
 Integrating over $M$ we get, thanks to Gauss-Bonnet formula ($\int_M K_X v_X=2\pi(2-\varepsilon\gamma)$), 
 $$\int_M|h_X|^2= 4\pi(\varepsilon\gamma-2) +4\int_M |H_X|^2+ 2c|X(M)|$$
 with $\varepsilon = 2$ if $M$ is orientable and $\varepsilon=1$ otherwise.
 Hence, 
 $$\lambda_1(L_{X,\alpha,\beta})\le -\frac 1 {a}\left( 4\alpha \pi(\varepsilon\gamma-2) +2c\alpha a+ (4\alpha+\beta) \int_M|H_X|^2\right).$$
Since $4\alpha+\beta\ge 0$, one can replace $\int_M|H_X|^2$ by its lower bound given by Lemma \ref{chen}, that is $4\pi-c a$. Thus,
 $$\lambda_1(L_{X,\alpha,\beta})\le -\frac 1 {a}\left[ 4\alpha \pi(\varepsilon\gamma-2) +2c\alpha a+ (4\alpha+\beta) (4\pi -ca)\right]$$ $$=\lambda_1(L_{S(a),\alpha,\beta})-\frac{4\varepsilon\pi}{a}\alpha \gamma.$$
 
If  $4\alpha+\beta> 0$, then the equality in (\ref{l1c}) implies the equality $\int_M|H_X|^2=4\pi -ca$ which implies that $X$ is a totally umbilic immersion (Lemma \ref{chen}). 

If $4\alpha+\beta= 0$, then the equality in (\ref{l1c}) holds if and only if the constant functions are first eigenfunctions of the operator $L_{X,\alpha,-4\alpha}=-\Delta_X +2\alpha K_X$ which implies that the Gaussian curvature $K_X$ is constant on $M$. 


  \end{proof}

 The main ingredient in the proof of results concerning the second eigenvalue is the \emph{conformal area} introduced by Li and Yau \cite{LY}
 as follows. Let $g$ be a Riemannian metric on $M$. For any integer $d\ge 2$ and any conformal map $f$ from $(M,g)$ to the $d$-dimensional standard sphere $\Sp^d$, we set

$$A_c(f)=\sup_{\phi\in G(d)}|\phi\circ f(M)|$$
where $G(d)$ is the group of conformal transformations of $\Sp^d$.
The conformal area $A_c(M,[g])$ of the conformal class of $(M,g)$  is defined as the infimum of  $A_c(f)$, where $f$ runs over the set of all conformal maps from $(M,g)$ to standard spheres of arbitrary dimensions.

 Li and Yau proved that if a conformal map $f:M\to \Sp^d$ is such that $f^{-1}(x)$ contains $k$ points for some $x\in \Sp^d$, then
 $$A_c(f)\ge 4\pi k.$$
 Consequently, 
 \begin{equation}\label{eq 1}
 A_c(M,[g])\ge 4\pi.
 \end{equation}
 On the other hand, they showed that the conformal area of a surface of genus $\gamma$ can be estimated from above in terms of $\gamma$. In \cite{EI}, Ilias and the author noticed that the lower bound given in \cite{LY} can be improved as follows:
 \begin{equation}\label{eq 2}
  A_c(M,[g])\le 
 \left\{ 
\begin{array}{ll}
4\pi\Big[\frac{\gamma + 3} 2\Big] \text{ if } M \text{ is orientable } \\
12\pi\Big[\frac{\gamma + 3} 2\Big] \text{ if } M  \text{ is non-orientable,}
\end{array}
\right\} 
\end{equation}
 where $[x]$ stands for the integer part of $x$.
 
   Li and Yau also proved that for surfaces in $\R^n$, the integral $\int_M |H_X|^2$ is bounded below by the conformal area. This result extends to surfaces in $N^n(c)$. Recall first that for all $c\in \R$, $N^n(c)$ admits a natural conformal map $\Pi_c:N^n(c)\to\Sp^n$ into the standard sphere $\Sp^n$. For instance, $\Pi_0$ can be taken as the inverse of a stereographic projection  and, using the model of the pseudo-sphere in the Minkowski space for $\Hyp^n$, we can take $\Pi_{-1}(x_0, x_1,\dots,x_n)=\frac1{x_0}(1,x_1,\dots,x_n)$.

 \begin{lemma}\label{li-yau} For any compact surface $X:M\to N^n(c)$ one has
 $$\int_M \left( |H_X|^2+c\right)v_X\ge A_c(\Pi_c\circ X)\ge A_c(M,[g_X]),$$
 where $g_X$ is the Riemannian metric induced by $X$ on $M$.  
 \end{lemma}
 \begin{proof}
 Since the Riemannian metric induced on $N^n(c)$ by $\Pi_c$ is conformally equivalent to the standard one, we have thanks to the invariance of $\int_M( |h_X|^2- 2|H_X|^2)$ under conformal changes of metric in the ambient space, 
 $$\int_M ( |h_X|^2- 2|H_X|^2)v_X=\int_M( |h_{\Pi_c\circ X}|^2- 2|H_{\Pi_c\circ X}|^2)v_{\Pi_c\circ X}.$$
 Moreover, for any conformal transformation $\phi$ of $\Sp^n$, the map $\phi\circ \Pi_c:N^n(c)\to\Sp^n$ is conformal and one has
  $$\int_M( |h_X|^2- 2|H_X|^2)v_X= \int_M( |h_{\tilde{X}}|^2 - 2|H_{\tilde{X}}|^2)v_{\tilde{X}},$$
 where $\tilde{X}=\phi\circ \Pi_c \circ X$. As in the proof of Lemma \ref{chen}, using Gauss equation and Gauss-Bonnet theorem we deduce
$$\int_M \left( |H_X|^2+c\right)v_X=\int_M \left(|H_{\tilde{X}}|^2 +1\right)v_{\tilde{X}}\ge |{\tilde{X}}(M)|.$$
Thus,
$$\int_M \left( |H_X|^2+c\right)v_X\ge \sup_{\phi\in G(d)}|\phi\circ \Pi_c\circ X(M)|=A_c(\Pi_c\circ X).$$

 \end{proof}
 
 In \cite{LY}, Li and Yau proved that the second eigenvalue of the Laplace-Beltrami operator is dominated in terms of the conformal area. Ilias and the author \cite{EI5} extended this result to Schrödinger type operators as follows.
 \begin{lemma}[\cite{EI5}]\label{conformalvol} For any continuous function $q$ on $M$, one has $$\lambda_2(- \Delta_X+q)\; |X(M)|\le 2 A_c(M,[g_X]) +\int_M q.$$
 \end{lemma}

 
 The proof of Theorem \ref{N(c)} follows from the following
  \begin{proposition}\label{prop1}  Let $X:M\to R^n$ be a compact orientable immersed surface of genus $\gamma$ such that $|X(M)|=a$ and let $\alpha\in\R$ and $\beta\in\R$ be such that $4\alpha+\beta\ge 0$.
  
\begin{itemize}
\item[(i)] 
  If $\gamma=0$, then
  $$\lambda_2(L_{X,\alpha,\beta})\le \lambda_2(L_{S(a),\alpha,\beta}).$$
	\item[(ii)] 
  If $4\alpha+\beta\ge 2$, then
  $$\lambda_2(L_{X,\alpha,\beta})- \lambda_2(L_{S(a),\alpha,\beta})\le - \frac{8\pi} a \alpha \gamma.$$
  \item[(iii)] If  $0\le 4\alpha+\beta\le 2$ and $\gamma$ is even, then
  $$\lambda_2(L_{X,\alpha,\beta})- \lambda_2(L_{S(a),\alpha,\beta})\le -\frac{4\pi} a  \frac \gamma 2 \left( 8\alpha +\beta -2\right) .$$
 \item[(iv)]  If  $0\le 4\alpha+\beta\le 2$ and $\gamma$ is odd, then
  $$\lambda_2(L_{X,\alpha,\beta})- \lambda_2(L_{S(a),\alpha,\beta})\le -\frac{4\pi} a  \frac {\gamma+1} 2 \left( 4\frac{2\gamma + 1}{\gamma + 1}\alpha +\beta -2\right).$$
  
\end{itemize}
  \end{proposition}

 \begin{proof}  
  
 Applying Lemma \ref{conformalvol} with $q=-\alpha |h_X|^2-\beta |H_X|^2$, one obtains 
  $$\lambda_2(L_{X,\alpha,\beta})\le \frac 1 {|X(M)|}\left( 2 A_c(M,[g_X]) -\alpha \int_M|h_X|^2-\beta \int_M|H_X|^2\right).$$ 
  With the same notations as in the proof of Proposition \ref{lam1}, we have
  $$\int_M|h_X|^2= 4\pi(\varepsilon\gamma-2) +4\int_M |H_X|^2+ 2c|X(M)|.$$
 Hence, 
 $$\lambda_2(L_{X,\alpha,\beta})\le \frac 1 {a}\left( 2 A_c(M,[g_X])- 4\alpha \pi(\varepsilon\gamma-2) -2c\alpha a- (4\alpha+\beta) \int_M|H_X|^2\right).$$
Since $4\alpha+\beta\ge 0$, we may apply Lemma \ref{li-yau} to get 
 \begin{equation}\label{eq 3}
 \lambda_2(L_{X,\alpha,\beta})\le \frac 1 {a}\left[ (2-4\alpha-\beta) A_c(M,[g_X])- 4\alpha \pi(\varepsilon\gamma-2) +(2\alpha + \beta) c a\right].
  \end{equation}
\noindent (i) Assume $\gamma=0$. Then $(M,g_X)$ is conformally equivalent to the standard sphere and we have $A_c(M,[g_X])=A_c(\Sp^2)=4\pi$. Thus,
  $$\lambda_2(L_{X,\alpha,\beta})\le \frac 1 {a}\left[ 4\pi(2-4\alpha-\beta) + 8\alpha \pi +(2\alpha + \beta) c a\right]$$
  $$=\frac{8\pi} a -(2\alpha+\beta)\left(\frac{4\pi} a - c\right)
 = \lambda_2(L_{S(a),\alpha,\beta}).$$
 
\noindent (ii) Assume $4\alpha+\beta\ge 2$. Then one can replace in (\ref{eq 3}) the conformal area $A_c(M,[g_X])$ by its universal lower bound $4\pi$ and obtain
$$\lambda_2(L_{X,\alpha,\beta})\le \frac 1 {a}\left[ 4\pi(2-4\alpha-\beta) -4\alpha \pi(\varepsilon\gamma-2) +(2\alpha + \beta) c a\right]$$
 $$=\lambda_2(L_{S(a),\alpha,\beta}) -\frac{4\pi} a \varepsilon\alpha\gamma.$$
 \noindent (iii) and (iv) Assume $0\le 4\alpha+\beta < 2$. Using  (\ref{eq 3}) and (\ref{eq 2}) we get in the orientable case ($\varepsilon=2$),
 $$\lambda_2(L_{X,\alpha,\beta})\le \frac 1 {a}\left( 4\pi(2-4\alpha-\beta) \Big[\frac{\gamma + 3} 2\Big]- 4\alpha \pi(2\gamma-2) +(2\alpha + \beta) c a\right) $$
 $$=  \lambda_2(L_{S(a),\alpha,\beta}) +\frac{4\pi} a \left(\left( 2-4\alpha-\beta\right) \Big[\frac{\gamma + 1} 2\Big] -2\alpha \gamma \right).$$
 Replacing $[\frac{\gamma + 1} 2] $ by $\frac \gamma 2$ when $\gamma $ is even and by $\frac{\gamma + 1} 2$ when $\gamma$ is odd, we get the desired inequalities.
  \end{proof}
  In the non-orientable case, the same arguments lead to the following
  \begin{proposition}\label{prop2}  Let $X:M\to N^n(c)$ be a compact non-orientable immersed surface of genus $\gamma$ such that $|X(M)|=a$ and let $\alpha\in\R$ and $\beta\in\R$ be such that $4\alpha+\beta\ge 0$.
  
\begin{itemize}
	\item[(i)] 
  If $4\alpha+\beta\ge 2$, then
  $$\lambda_2(L_{X,\alpha,\beta})- \lambda_2(L_{S(a),\alpha,\beta})\le - \frac{4\pi} a \alpha \gamma.$$
  \item[(ii)] If  $0\le 4\alpha+\beta\le 2$ and $\gamma$ is even, then
  $$\lambda_2(L_{X,\alpha,\beta})- \lambda_2(L_{S(a),\alpha,\beta})\le -\frac{4\pi} a  \frac {3\gamma+4} 2 \left( (4+\frac{2\gamma}{3\gamma+4})\alpha +\beta -2\right) .$$
 \item[(iii)]  If  $0\le 4\alpha+\beta\le 2$ and $\gamma$ is odd, then
  $$\lambda_2(L_{X,\alpha,\beta})- \lambda_2(L_{S(a),\alpha,\beta})\le -\frac{4\pi} a  \frac {3\gamma+5} 2 \left( (4+\frac{2\gamma}{3\gamma+5})\alpha +\beta -2\right) .$$
 
\end{itemize}
  \end{proposition}
  \begin{proof}[Proof of Theorem \ref{genus=0}, Theorem  \ref{genus=arb} and Theorem \ref{N(c)} (II)] Let $X:M\to N^n(c)$ be a compact immersed surface.
  When $M$ is of genus zero, Proposition \ref{prop1}(i) gives, for all $\alpha\in\R$ and $\beta\in\R$ such that $4\alpha+\beta\ge 0$, 
  $$\lambda_2(L_{X,\alpha,\beta})\le \lambda_2(L_{S(a),\alpha,\beta}).$$
  Moreover, from the proof of this proposition, we see that if the equality holds with $4\alpha+\beta >0$, then $\int_M \left( |H_X|^2+c\right)v_X= A_c(M,[g_X]) =4\pi$ which implies that $X(M)$ is totally umbilic (Lemma \ref{chen}). 
  This completes the proof of the first part of Theorem \ref{N(c)} (II) and Theorem \ref{genus=0}. 
  
  Assume now the genus $\gamma$ of $M$ is even, $\gamma \neq 0$, and $8\alpha+\beta > 2$. Then either $4\alpha+\beta \ge 2$ or  $0\le 4\alpha+\beta <2 $. In both cases, Proposition \ref{prop1} gives the strict inequality  $\lambda_2(L_{X,\alpha,\beta})< \lambda_2(L_{S(a),\alpha,\beta})$. The same argument works in the case where $\gamma$ is odd. Thus, Theorem  \ref{genus=arb} as well as the last part of Theorem \ref{N(c)} (II) are proved.

 \end{proof}
 
 We end this discussion with the following result concerning non-orientable surfaces of the lowest genus. Indeed, such a surface is homeomorphic to the real projective plane whose genus is 1. The Veronese surface $V:\R P^2\to \R^6$, naturally embedded in $\R^6$, is given by
 $$V(x,y,z)=(x^2,y^2,z^2, \sqrt 2 x y, \sqrt 2 x z, \sqrt 2 y z),$$
 where $\R P^2$ is considered as quotient of the standard sphere. This surface lies as a minimal surface in the 4-dimensional sphere (of radius $\sqrt \frac 2 3$ centered at $(\frac 1 3,\frac 13, \frac 13)$)  obtained as the intersection of $\Sp^5$ with the hyperplane $\{ X+Y+Z=1\}$. It has constant Gaussian curvature $K_V= \frac 1 2 $ and area $|V(\R p^2)|=4\pi$. Moreover, $|H_V|^2=\frac 3 2$, $|h_V|^2=5$ and $\lambda_2(-\Delta_V)=3$. Thus
 $$\lambda_1(  -\Delta_V-\alpha|h_V|^2-\beta|H_V|^2)=- 5\alpha +\frac3 2 \beta $$ 
 and
 $$\lambda_2(  -\Delta_V-\alpha|h_V|^2-\beta|H_V|^2)=3 - 5\alpha +\frac3 2 \beta. $$ 
   
  \begin{proposition}\label{veronese} Let $M$ be a non-orientable compact surface of genus one (i.e. homeomorphic to $\R P^2$) and let  $\alpha\in\R$ and $\beta\in\R$ be such that $4\alpha+\beta\ge 0$. For any immersion $X:M\to \R^n$ such that $|X(M)|=4\pi$, one has 
$$
 \lambda_1(L_{X,\alpha,\beta})\le \lambda_1(L_{V,\alpha,\beta})
$$
  and
 $$
 \lambda_2(L_{X,\alpha,\beta})\le \lambda_2(L_{V,\alpha,\beta})
$$
 \end{proposition}
 
 \begin{proof}
 The real projective plane carries only one conformal class of metrics with $A_c(M,[g_X])=6\pi$.   
Following exactly the same steps as in the proof of Theorem \ref{N(c)} above, we get the desired inequalities.

  \end{proof}


\bibliographystyle{plain}
\bibliography{bibli}
\end{document}